\UseRawInputEncoding
\documentclass[11pt]{article}

\usepackage{tikz}
\usepackage{subfigure}
\usepackage[english]{babel}
\usepackage{graphicx}
\usepackage{float}

\usepackage[center]{caption2}
\usepackage{amsfonts,amssymb,amsmath,latexsym,amsthm}
\usepackage{multirow}
\usepackage{mathrsfs}
\usepackage{indentfirst}

\newtheorem{thm}{Theorem}[section]

\newtheorem{lem}[thm]{Lemma}

\newtheorem{obs}[thm]{Observation}

\newtheorem{clm}[thm]{Claim}

\begin{document}
\title{The degree and codegree threshold for linear triangle covering in 3-graphs
\thanks{This work was supported by the National Natural Science Foundation of China (No. 12071453), the National Key R and D Program of China (2020YFA0713100),  Anhui Initiative in Quantum Information Technologies (AHY150200), and the Innovation Program for Quantum Science and Technology (2021ZD0302904).}
}
 \author{Yuxuan Tang$^a$, \quad Yue Ma$^a$, \quad Xinmin Hou$^{a,b}$\\
\small $^a$School of Mathematical Sciences\\
\small University of Science and Technology of China, Hefei, Anhui 230026, China.\\
\small $^{b}$ CAS Key Laboratory of Wu Wen-Tsun Mathematics\\
\small University of Science and Technology of China, Hefei, Anhui 230026, China.\\
\small yx0714@mail.ustc.edu.cn \quad mymy@mail.ustc.edu.cn \quad xmhou@ustc.edu.cn}

\date{}

\maketitle

\begin{abstract}
Given two $k$-uniform hypergraphs $F$ and $G$, we say that $G$ has an $F$-covering if every vertex in $G$ is contained in a copy of $F$.
For $1\le i \le k-1$, let $c_i(n,F)$ be the least integer such that every $n$-vertex $k$-uniform hypergraph $G$ with $\delta_i(G)> c_i(n,F)$ has an $F$-covering. 
The covering problem has been systematically studied by Falgas-Ravry and Zhao [Codegree thresholds for covering 3-uniform hypergraphs,
SIAM J. Discrete Math., 2016].   Last year, Falgas-Ravry, Markstr\"om, and Zhao [Triangle-degrees in graphs and tetrahedron coverings
in 3-graphs, Combinatorics, Probability and Computing, 2021] asymptotically determined $c_1(n, F)$ when $F$ is the generalized triangle.
In this note, we give the exact value of $c_2(n, F)$ and  asymptotically determine $c_1(n, F)$ when $F$ is the linear triangle $C_6^3$, where $C_6^3$ is the 3-uniform hypergraph with vertex set $\{v_1,v_2,v_3,v_4,v_5,v_6\}$ and edge set  $\{v_1v_2v_3,v_3v_4v_5,v_5v_6v_1\}$.
\end{abstract}

\section{Introduction.}	
Given a positive integer $k\ge 2$, a {\it$k$-uniform hypergraph} (or a $k$-graph) $G=(V,E)$ consists of a vertex set $V=V(G)$ and an edge set $E=E(G)\subset \binom{V}{k}$, where $\binom{V}{k}$ denotes the set of all $k$-element subsets of $V$. We write graph for $2$-graph for short. Let $G=(V,E)$ be a simple $k$-graph. For any $S\subseteq V(G)$, let $N_{G}(S)=\{T\subseteq V(G)\backslash S: T\cup S\in E(G)\}$ and the degree $d_{G}(S)=|N_{G}(S)|$. For $1\le i\le k-1$, the {\it minimum $i$-degree} of $G$, denoted by $\delta_i(G)$, is the minimum of $d_G(S)$ over all $S\in\binom{V(G)}{i}$. We also call $\delta_{k-1}(G)$ the minimum codegree of $G$ and $\delta_{1}(G)$ the minimum degree of $G$. The $link\; graph$ of a vertex $x$ in $V$, denoted by $G_x$, is a $(k-1)$-graph $G_x=\{V(G)\backslash \{x\},N_G(x)\}$. 

Given a $k$-graph $F$, we say a $k$-graph $G$ has an $F$-covering if each vertex of $G$ is contained in some copy of $F$. For $1\le i \le k-1$, define
$$c_i(n,F)=\max\{\delta_i(G):G\;\text{is\;a\;$k$-graph\;on\;$n$\;vertices\;with\;no\;$F$-covering}\}.$$ 

For graphs $F$, the $F$-covering problem was solved asymptotically in \cite{ref2} by showing that $c_1(n,F)=\left( \frac{\chi(F)-2}{\chi(F)-1}+o(1)\right) n$, where $\chi(F)$ is the chromatic number of $F$.  Falgas-Ravry and Zhao \cite{ref4} initiated the study of the covering problem in $3$-graphs, and determined the exact value of $c_2(n,K_4^{3})$  (where $K^{3}_r$ denotes the complete 3-graph on $r\ge 3$ vertices) for $n>98$ and gave
bounds on $c_2(n, F)$ which are apart by at most 2 in the cases where $F$ is $K_4^{3-}$ ($K_4^{3}$ with one edge
removed, also called a {\em generalized triangle}), $K_5^{3−}$, and the tight cycle $C_5^{3}$ on 5 vertices. Yu, et al~\cite{ref5} showed that $c_2(n,K_4^{3-})=\lfloor \frac{n}{3} \rfloor$, and $c_2(n,K_5^{3-})=\lfloor \frac{2n-2}{3} \rfloor$. 
Last year, Falgas-Ravry, Markstr\"om, and Zhao~\cite{FMZh21} gave close to optimal bounds of $c_1(n, K^{(3)}_4)$ and asymptotically determined $c_1(n, K_4^{(3)-})$. 
There are some other related results in literature, for example in~\cite{FT22, FZ22}.

A {\em linear triangle} $C_6^3$ is a 3-graph with vertex set $\{v_1,v_2,v_3,v_4,v_5,v_6\}$ and edge set  $\{v_1v_2v_3,v_3v_4v_5,v_5v_6v_1\}$. Gao and Han showed in \cite{ref3} that, when $n\in 6\mathbb{Z}$ is sufficiently large, and $H$ is a 3-graph on $n$ vertices with $\delta_2(H)\ge n/3$, then $H$ has a $C_6^3$-covering such that every vertex $H$ is covered by exactly one $C_6^3$ (also called a $C_6^3$-factor or a perfect $C_6^3$-tiling). In this article, we determine the exact value of $c_2(n, C_6^3)$ and an asymptotic optimal value of $c_1(n, C_6^3)$. The main results are listed as follows. 

%Below, we give our main result:  
\begin{thm}\label{THM: main1}
	For $n\ge 6$, $c_2(n, C_6^3)=1$. 
\end{thm}

\begin{thm}\label{THM: main2}
	For $n\ge 6$, $\frac{3-2\sqrt{2}}{4}n^2-n < c_1(n, C_6^3) < \frac{3-2\sqrt{2}}{4}n^2+3n^\frac{3}{2}$.
\end{thm}

\noindent{\bf Remark:} We believe that  $c_1(n, C_6^3)=\frac{3-2\sqrt{2}}{4}n^2-O(n)$.  

{The rest of the article is arranged as follows. In Section 2, we construct extremal graphs with minimum codegree one that has no $C_6^3$-covering, and minimum degree greater than  $\frac{3-2\sqrt{2}}{4}n^2-n$ that has no $C_6^3$-covering, respectively. The proofs of Theorems~\ref{THM: main1} and~\ref{THM: main2} are given in Sections 3 and 4, respectively.}

\section{Constructions}
We introduce two constructions involving our result. For two families of sets $\mathcal{A}$ and $\mathcal{B}$, define $\mathcal{A}\vee \mathcal{B}=\{A\cup B : A\in\mathcal{A} \text{ and } B\in\mathcal{B}\}$. 

\noindent\textbf{Construction 1: } Let $G_1=(V_1,E_1)$ be a 3-graph with $V_1=\{x\}\cup V'$ and 
$$E_1=\{\{x\}\}\vee {V'\choose 2}.$$
The following observation can be checked directly.
\begin{obs}\label{OBS: G1} 
	$\delta_2(G_1)=1$ and $G_1$ has no $C_6^3$-covering. 
\end{obs}

\noindent\textbf{Construction 2: } Let $G_2=(V_2,E_2)$ be a 3-graph with $V_2=\{u\}\cup A_1\cup A_2\cup B_1\cup B_2$, and 
$$E_2=\left(\left\{\{u\}\right\}\vee{A_1\choose 1}\vee{A_2\choose 1}\right)\cup\left({A_1\choose 1}\vee{B_1\choose 2}\right)\cup\left({A_2\choose 1}\vee{B_2\choose 2}\right)\cup {B_1\cup B_2 \choose 3},$$
%\left({B_1\choose 1}\vee{B_2\choose 2}\right)\cup\left({B_1\choose 2}\vee{B_2\choose 1}\right)\cup{B_1\choose 3}\cup {B_2\choose 3}\right\}$
%$E_2=\{ua_1a_2:a_1\in A_1,a_2\in A_2\} \cup \{b_ib_ja_k:b_i,b_j\in B_1,a_k\in A_1\} \cup \{b_sb_ta_r:b_s,b_t\in B_2,a_r\in A_2\} \cup \{b_1b_2b_3:b_1,b_2,b_3\in B_1\} \cup \{b_4b_5b_6:b_4,b_5,b_6\in B_2\} \cup \{b_lb_mb_n:b_l,b_m\in B_1,b_n\in B_2\} \cup \{b_wb_xb_y:b_w,b_x\in B_2,b_y\in B_1\} $, 
where $\left| |A_1|-|A_2|\right| \le 1$, $|B_1|=|B_2|=\lfloor (1-\frac{\sqrt{2}}{2})n\rfloor $. 
\begin{figure}[H]
\centering	 \includegraphics[scale=0.5]{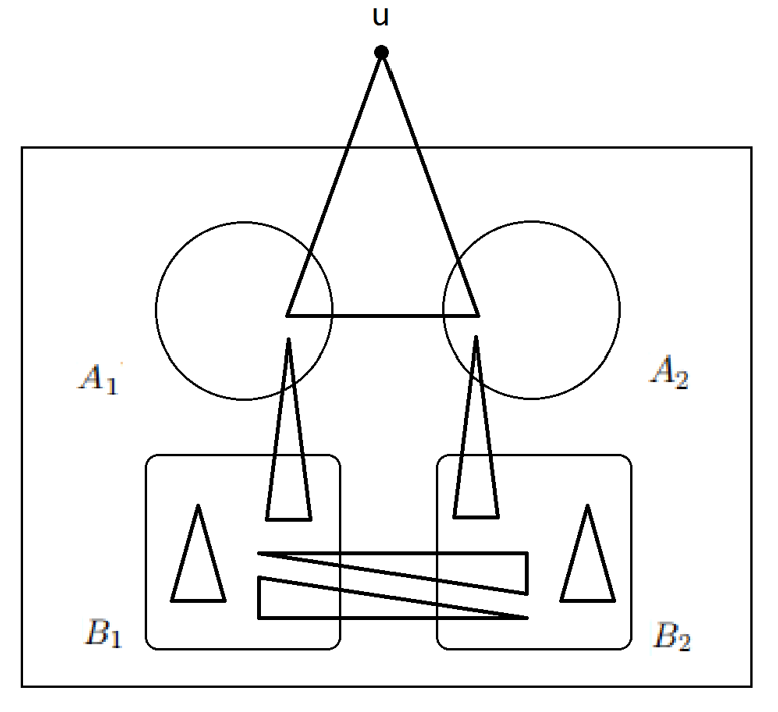}
	\caption{Construction 2}
\end{figure}
\begin{obs}\label{OBS: G2}
	$\delta_{1}(G_2) > \frac{3-2\sqrt{2}}{4}n^2-n$ and $G_2$ has no $C_6^3$ covering $u$. 
\end{obs}
\begin{proof}
 It is easy to observe that $G_2$ has no $C_6^3$ covering $u$.  Let $b=|B_1|=|B_2|=\lfloor (1-\frac{\sqrt{2}}{2})n\rfloor$ and $a=\frac{n-1-2b}{2}\ge \frac{(\sqrt{2}-1)n-1}{2}$. 
Without loss of generality, assume $|A_1|=\lfloor a \rfloor $ and $|A_2|=\lceil a \rceil $.
Now let us calculate $\delta_1(G_2)$.
When $n\le 23$, we have $\frac{3-2\sqrt{2}}{4}n^2-n < 0$. So we may suppose $n\ge 24$. Therefore,  $b\ge 2$. We have $|A_1|+|A_2|=n-1-2b$. 
Choose $v\in V(G_2)$.

If $v=u$, then $$d_{G_2}(v)=\lfloor a \rfloor \cdot \lceil a \rceil \ge a^2-\frac{1}{4}\ge \frac{3-2\sqrt{2}}{4}n^2-\frac{(\sqrt{2}-1)n}2> \frac{3-2\sqrt{2}}{4}n^2-n.$$
If $v\in A_1\cup A_2$, then $$d_{G_2}(v)\ge \lfloor a \rfloor +\binom{b}{2}\ge \frac{(\sqrt{2}-1)n-3}{2}+\frac{((1-\frac{\sqrt{2}}{2})n-1)((1-\frac{\sqrt{2}}{2})n-2)}{2} > \ \frac{3-2\sqrt{2}}{4}n^2-n.$$
 If $v\in B_1\cup B_2$, then $$d_{G_2}(v)\ge (b-1)\cdot \lfloor a \rfloor +\binom{2b-1}{2}\ge \lfloor a \rfloor + \binom{b}{2} > \frac{3-2\sqrt{2}}{4}n^2-n.$$ Therefore, $\delta(G_2)>\frac{3-2\sqrt{2}}{4}n^2-n$. 
	
\end{proof}

\section{Proof of Theorem~\ref{THM: main1}}
We first introduce a lemma, which is of great importance to our proof. 
Let $P_k$ (resp. $C_k$) denote a path (resp. a cycle) with $k$ vertices. 

\begin{lem}\label{LEM: lem3.1}
	Let $G$ be a 3-graph with $\delta _2 (G) \ge 2$ and $v\in V(G)$.  If $v$ is not covered by any $C_6^3$, then $G_v$ must be $P_5$-free and $2P_3$-free. 
\end{lem}

\begin{proof}
In contrast, let us first suppose that $G_v$ contains a $2P_3$, say, $w_1u_1w_2$ and $w_3u_2w_4$. Since $\delta _2 (G) \ge 2$, there exists a vertex $x \ne v$, $u_1u_2x\in E(G)$. 
	Since $\{w_1,u_1,w_2\}\cap\{w_3, u_2, w_4\}=\emptyset$, we may choose $w_2$ and $w_3$ that are different from $x$. Then the subgraph induced by $\{ v,w_2,u_1,x,u_2, w_3\}$ contains a $C_6^3$. This gives the conclusion that  $G_v$ does not contain a $2P_3$.
%	If $x\notin \{w_1,w_2,w_3,w_4\}$, then it is obvious that $\{ v,w_1,u_1,x,u_2, w_3\}$ forms a $C_6^3$, where the edges are $vw_1u_1$, $u_1xu_2$, $u_2w_3v$; if $x\in \{w_1,w_2,w_3,w_4\}$, without loss of generality, we may assume that $x = w_1$, then $\{ v,w_2,u_1,x,u_2, w_3\}$ forms a $C_6^3$, where the edges are $vw_2u_1$, $u_1xu_2$, $u_2w_3v$. This gives the conclusion that  $G_v$ does not contain a $2P_3$.
	
	%For the condition of $P_5$, similarly 
	Now suppose that $G_v$ contains a $P_5$, say, $w_1u_1wu_2w_2$. Then similarly, there exists a vertex $x \ne v$, $u_1u_2x\in E(G)$. If $x \notin \{w_1,w_2\}$, then the subgraph induced by  $\{ v,w_1,u_1,x,u_2, w_3\}$ contains a $C_6^3$.
	%, where the edges are $vw_1u_1$, $u_1xu_2$, $u_2w_3v$; if $x=w$, then $vw_1u_1$, $u_1wu_2$, $u_2w_2v$ forms a $C_6^3$;
	If $x \in \{w_1,w_2\}$, without loss of generality, assume $x=w_1$, then the subgraph induced by $\{v, w, u_1, w_1, u_2, w_2\}$ contains a $C_6^3$. This gives the conclusion that $G_v$ is $P_5$-free. 
	%By the two conditions above, we may achieve the result of our lemma. 
\end{proof}

Now we are ready to finish the proof of Theorem~\ref{THM: main1}. 
\begin{proof}[Proof of Theorem~\ref{THM: main1}]
	Suppose to the contrary that there is a 3-graph $G$ with $\delta_2(G)\ge 2$ and a vertex $v\in V(G)$ that is not covered by $C_6^3$.
	Since $\delta _2 (G) \ge 2$, consider any vertex $u \in V(G_v)$, $|N_G(u,v)|\ge 2$, which means $|N_{G_v}(u)|\ge 2$. Therefore,  $\delta (G_v) \ge 2$.
	%then we can at least find another two vertices $w$ and $x$, such that $uvw$ and $uvx$ form two 3-edges of $G$, so $w$ and $x$ are neighbors of u in $G_v$, which means $\delta (G_v) \ge 2$. 
	
	By Lemma~\ref{LEM: lem3.1},  $G_v$ is $P_5$-free and $2P_3$-free. Therefore, the longest path in $G_v$ must be $P_4$ or $P_3$ (otherwise, $G_v$ must be a matching, which contradicts the fact that  $\delta(G_v)\ge 2$).
%	 here because if there only exists $P_2$, then every component of $G_v$ contains at most two vertices, which leads to $\delta (G_v)=1$, this contradicts to the condition that  $\delta (G_v) \ge 2$). 
	
Note that every component of $G_v$ contains a cycle since $\delta(G_v)\ge 2$. 	If the longest path is $P_3$ in $G_v$, then every component of $G_v$ must be a $K_3$. %let us consider one of the components of $G_v$, which contains at least 3 vertices. If it is not $K_3$, in order that there isn't any path longer than 3 in $G_v$, this component must be a star, but this contradicts to $\delta (G_v) \ge 2$. 
Since $|V(G_v)|\ge 5$, we have $G_v\cong tK_3$ for some $t\ge 2$. However, this contradicts the fact that $G_v$ is $2P_3$-free. 
	
Now suppose the longest path is $P_4$ in $G_v$.  Consider the component $H$ that contains $P_4=u_1w_1w_2u_2$. 
%Obviously, any other vertex in this component can't be adjacent to $u_1$ and $u_2$, otherwise a $P_5$ is formed. 
Since $P_4$ is a longest path and $d_{H}(u_1) \ge 2$ and $d_{H}(u_2) \ge 2$, it can be easily checked that $C_4\subseteq H$, and thus, $V(H)=V(P_4)$ (otherwise, $H$ contains a $P_5$, a contradiction). Since $|V(G_v)|\ge n-1\ge 5$, $G_v$ has at least another component $H'$. Because every component of $G_v$ contains a cycle, we have $|V(H')|\ge 3$, a contradiction to the fact that $G_v$ is $2P_3$-free. 
%	\\
%	\textbf{Case 1:} $u_1$ is adjacent to $u_2$, which gives that $u_1w_1w_2u_2$ forms a $C_4$. Now, this component contains only four vertices, or a $P_5$ can be formed. In consideration of the lack of vertices, there exists other components, whose structure contain $P_3$ since $\delta (G_v) \ge 2$. Now we obtain a $2P_3$ in $G_v$, a contradiction; 
%	\\
%	\textbf{Case 2:} $u_1$ is adjacent to $w_1$ and $u_2$ is adjacent to $w_2$, similarly because of the lack of the vertices, there still exists another component that contains a $P_3$, a contradiction.
%	\\
%	Synthesize all the conditions above, it is clear that the hypothesis can't be real. And this finishes the proof of our main theorem. 
\end{proof}

\section{Proof of Theorem~\ref{THM: main2}}
We need the fundamental result in extremal graph theory due to Mantel~\cite{M07} and Tur\'an~\cite{ref6}. 
\begin{thm}[Mantel, 1907, Tur\'an, 1941]\label{THM: Turan} 
Let $r\ge 2$ and $T_{n,r}$ be the Tur\'an graph, i.e., the complete $r$-partite graph on $n$ vertices with the size of each part differing by at most one.
%	$$\ex(n, K_{r+1})=e(T_{n,r}),$$ 
If $G$ is a graph containing no $K_{r+1}$ as a subgraph, then $e(G)\le e(T_{n,r})$.
\end{thm} 
Now we give the proof of Theorem~\ref{THM: main2}.
\begin{proof}[Proof of Theorem~\ref{THM: main2}:]
	The lower bound of $c_1(n, C_6^3)$ is a direct corollary of Observation~\ref{OBS: G2}. 
Therefore, it is sufficient to show that every 3-graph $G$ on $n$ vertices with $\delta_1(G)\ge  \frac{3-2\sqrt{2}}{4}n^2+3n^\frac{3}{2}$ has a $C_6^3$-covering.

Suppose to the contrary that there is a 3-graph $G$ on $n$ vertices with $\delta_1(G)\ge  \frac{3-2\sqrt{2}}{4}n^2+3n^\frac{3}{2}$ and a vertex $u\in V(G)$ that is not contained in a copy of $C_6^3$. 
%We consider the link graph of $u$ which denoted as $G_u=\{V',E\}$,where $V'=V \setminus u$, and $E$ contains all the edges $v_iv_j$ that form a 3-edge with $u$ in $G$. We say a vertex $w$ is $isolated$ in $G_u$ if $w$ is not adjacent to any vertex in $G_u$, and an edge $w_pw_q$ is $isolated$ in $G_u$ if $w_p$ and $w_q$ is not adjacent to any other vertices in $G_u$ except each other. 
%Recall that $G_u$ is the link graph of $u$. 
Recall that the link graph $G_u$ is the graph with vertex set $V'=V \setminus\{u\}$ and edge set $E=\{vw : uvw\in E(G)\}$. 
%Clearly, the link graph $G_u\subseteq H$ and $E(H)=E(G_u)$.
Denote $M_0$ as the set of components isomorphic to $K_2$ of $G_u$ and let $I_0$ be the set of components isomorphic to $K_1$ in $G_u$.
For any $v\in V(G_u)$, let $M(v)$ denote the set of components isomorphic to $K_2$ in $G_u-\{v\}-V(M_0)$, and let $I(v)$ denote the set of components isomorphic to $K_1$ in $G_u-\{v\}-I_0$.  Note that $M_0\cap M(v)=\emptyset$ and $I_0\cap I(v)=\emptyset$ by the definitions.
%Define a vertex $v\in V(G_u)$ to be a {\em good} vertex if $|I(v)|< n^\frac{1}{2}$; a {\em bad} vertex, otherwise.
 
%    We may first introduce a lemma: 
    \begin{clm}\label{CLM: 4.1}
%		For any ``good" vertex $x$ in $G_u$, if $d_{G_u}(x) \ge 4$ and $u$ is not covered by any $C_6^3$, then only two kinds of pairs in $\binom{V(G)}{2}$ can formulate a $3$-edge in $G$ with $x$: \\
%		1. The original edges of $M_0$ and $M(x)$; \\
%		2. The pair from the set $I_0\cup I(x)$. 
Let $v$ be a vertex in $G_u$ with $d_{G_u}(v)\ge 4$. Then 
$$E(G_v-u)\subseteq M_0\cup M(v)\cup{I_0\cup I(v)\choose 2}.$$  
	\end{clm}
    \begin{proof}
Choose $v_1v_2\in E(G_v-u)$. If $v_1v_2\notin M_0\cup M(v)\cup{I_0\cup I(v)\choose 2}$, then there is at least one of $v_1, v_2$, say $v_1$, that is not in $I_0\cup I(v)$. Thus $v_1\in V(M_0)\cup V(M(v))$. Therefore, there is a vertex $v_1'$ different  from $v_2$ such that $v_1v_1'\in M_0\cup M(v)$.
Since $d_{G_u}(v)\ge 4$, there is another vertex $v'$ different from $v_1, v_1'$ and $v_2$ such that $vv'\in E(G_u)$. Therefore, the edges $uv_1'v_1$, $v_1v_2v$ and $vv'u$ form a copy of $C_6^3$ covering $u$ in $G$, a contradiction.

    \end{proof}    
%    Now let us continue our proof of the main theorem.% We assume that all the vertices mentioned below have the same property——they are all ``good" vertices and their degree are all more than 4.  %According to the lemma, for some $v_1$, we have
%	$$d_G(v_1) \le d_{G_u}(v_1)+\frac{|M_0|+|M(v_1)|}{2}+\left|N_G(v_1)\cap \binom{I_0\cup I(v_1)}{2}\right|. $$
%	where $d\left( v_1,\binom{I_0\cup I(v_1)}{2}\right) $ denotes the 3-edges in the form of $v_1j_1j_2$, and here $j_1$ and $j_2$ are chosen from $I_0\cup I(v_1)$ arbitrarily.
%In this fomula, $d_{G_u}(v_1)+\frac{|M_0|+|M(v_1)|}{2}\le \frac{3}{2}n$, so we will focus on $d\left( v_1,\binom{I_0\cup I(v_1)}{2}\right)$. 

Now, define a vertex $v\in V(G_u)$ to be a {\em good} vertex if $|I(v)|< n^\frac{1}{2}$; a {\em bad} vertex, otherwise.	
For a good vertex $v$, a vertex $w$ in $G_v$ is called {\em a private} vertex of $v$ if $w\in I_0\cup I(v)$ and $d_{G_v-u}(w)\ge 2$. Let $X(v)$ denote the set of all private vertices of $v$, i. e. 
	$$ X(v)=\{ w\in V(G_v) : w\in I_0\cup I(v) \text{ and } d_{G_v-u}(w)\ge 2\}. $$ 
	
   For a good vertex $v$ with $d_{G_u}(v)\ge 4$, let $x=|X(v)|$ and let $J(v)=(I_0\cup I(v))\backslash X(v)$. Denote $H=G_{v}-u$.   By Claim~\ref{CLM: 4.1}, we have
    $$d_G(v) \le d_{G_u}(v)+|M_0|+|M(v)|+\left|E(H)\cap\binom{I_0\cup I(v)}{2}\right|. $$
    By the definition of $X(v)$, $d_H(w)\le 1$ for any $w\in J(v)$. Thus $\left|E(H)\cap\binom{I_0\cup I(v)}{2}\right|\le \sum_{w\in J(v)}d_H(w)+|{X(v)\choose 2}|\le |I_0|+|I(v)|-x+{x\choose 2}$. 
    %So $\left|E(H)\cap\binom{I_0\cup I(v_1)}{2}\right|\le \frac{|I_0|+|I_v|-x_1}2+{x_1\choose 2}$.
   % they contribute at most $n-x_1$ edges since the number of edges must be less than the summation of the degree of all vertices..
  Clearly, $d_{G_u}(v)\le n$ and $|M_0|+|M(v)|\le\frac{n-1-|I_0|-|I(v)|}2$. 
  %Let $J(v_1)=I_0\cup I(v_1)\backslash X(v_1)$. By the definition of private vertices, all the vertices in $J(v_1)$ has degree at most 1 in $G_{v_1}-u$. So they contribute at most $n-x_1$ edges since the number of edges must be less than the summation of the degree of all vertices. So} 
  Therefore,   $$d_{G_u}(v) <2n + \binom{x}{2}. $$
Note that  
    $$\frac{3-2\sqrt{2}}{4}n^2+3n^{\frac{3}{2}}\le \delta_1(G) \le d_G(v) < 2n + \binom{x}{2}.$$
Therefore, we have  
\begin{equation}\label{EQN: e1}  
	 |X(v)|=x\ge \frac{1+\sqrt{(2-\sqrt{2})^2n^2+24n^\frac{3}{2}-12n+1}}{2}\ge\left(1-\frac{\sqrt{2}}{2}\right)n+n^\frac{1}{2}.
\end{equation}
    
    Now, let us compute the number of edges of $G_u$. Define a {\em bad edge} in $G_u$ as an edge that is adjacent to at least one bad vertex, and a {\em good edge} if its two ends are good. Let $E_1$ denote the set of bad edges in $G_u$,
    % which are connected to at least one ``bad" vertex, 
    let $E_2$ be the set of good edges with one end of degree at most 3, and let $E_3$ denote the remaining edges in $G_u$.
    %formed by ``good" vertices with their degree at least 4. 
    Then $|E(G_u)|=|E_1|+|E_2|+|E_3|$. 
    By the definition of $E_2$, we have $|E_2|\le 3n$.
    Note that $I(v_1)\cap I(v_2)=\emptyset$ for different vertices $v_1, v_2\in V(G_u)$. 
Since $|I(v)|\ge n^{\frac 12}$ for a bad vertex $v\in V(G_u)$,  the number of the bad vertices is at most  $n/n^\frac{1}{2}=n^\frac{1}{2}$.
Therefore, they will contribute at most $n^\frac{3}{2}$ edges. That is, $|E_1|\le n^\frac{3}{2}$.    
%    {\color{red} For ``bad" vertices, their {\em generating isolated vertex sets} are disjoint. Since the size of the isolated vertex set of a ``bad" vertex will be at least $n^\frac{1}{2}$, so the number of the ``bad" vertices is at most  $n/n^\frac{1}{2}=n^\frac{1}{2}$,} so they will contribute at most $n^\frac{3}{2}$ edges. That is to say, $|E_1|\le n^\frac{3}{2}$.
% For ``good" vertices whose degree are at most 3, at most $3n$ edges will form since they can only form edges among themselves. This means $|E_2|\le 3n$. 
Since $n^\frac 32+3n+|E_3|=|E_1|+|E_2|+|E_3|=d_1(u)\ge\delta_1(G)\ge \frac{3-2\sqrt{2}}{4}n^2+3n^{\frac{3}{2}}$, we have  $E_3\not=\emptyset$.   
%    \color{red}Now, except for the cases above, for ``good" vertices with their degree at least 4, if there exists no pair of such ``good" vertices that are adjacent, then $E_3=\emptyset$, so $|E(G_u)|=|E_1|+|E_2|+|E_3|\le 3n+n^\frac{3}{2}$, and this means the number of edges of $G_u$ will be not enough, a contradiction. So there exists such ``good" vertices that are adjacent. \\
%    We claim as below: \\
	\begin{clm}\label{CLM: 4.2}
		 For any $e=v_1v_2\in E_3$, $X(v_1)\cap X(v_2)=\emptyset$. 
\end{clm}
\begin{proof}
Suppose to the contrary that there exists $w\in X(v_1)\cap X(v_2)$. Since $d_{G_{v_i}-u}(w)\ge 2$ for $i=1,2$, there are two different vertices $y_1, y_2$ with $y_1w\in E(G_{v_1}-u)$ and $y_2w\in E(G_{v_2}-u)$. Therefore, $v_1y_1w$, $wy_2v_2$ and $v_2uv_1$ form a copy of $C_6^3$, a contradiction.

	 %   {\color{red}On the contrary, suppose that there exists $w\in X(v_1)\cap X(v_2)$. Since its degree is at least 2, let us suppose $y_1$ is adjacent to $w$ in $G_{v_1}$, $y_2$ is adjacent to $w$ in $G_{v_2}$, $y_1\neq y_2$. then $v_1y_2w$, $wy_2v_2$ and $v_2uv_1$ will form a $C_6^3$, a contradiction.}
\end{proof}
Let $e=v_1v_2\in E_3$ and $x_i=|X(v_i)|$ for $i=1,2$. By (\ref{EQN: e1}) and Claim~\ref{CLM: 4.2}, we have 
$$2\left(1-\frac{\sqrt{2}}{2}\right)n+2n^\frac{1}{2}\le x_1+x_2=|X(v_1)\cup X(v_2)|\le |I_0\cup I(v_1)\cup I(v_2)|<|I_0|+2n^\frac 12,$$
the last inequality holds since $|I(v)|<n^\frac 12$ for any good vertex $v\in V(G_u)$.
Therefore, $|I_0|> 2(1-\frac{\sqrt{2}}{2})n$.

If $G_u[E_3]$ contains no $K_3$, then, by Theorem~\ref{THM: Turan}, $$|E_3|\le e(T_{n-1-|I_0|,2})\le\frac{(n-1-|I_0|)^2}{4}<\frac{\left((\sqrt{2}-1)n-1\right)^2}{4}.$$ 
Therefore,
$$d_1(u)=|E(G_u)|=\sum_{i=1}^3|E_i|<\frac{\left((\sqrt{2}-1)n-1\right)^2}{4}+3n+n^\frac{3}{2}\le \frac{3-2\sqrt{2}}{4}n^2+3n^\frac{3}{2}, $$
a contradiction. 
%and again, the fomula $3n+n^\frac{3}{2}$ here denote the number of edges from $E_1$, $E_2$. This also contradicts to the edge size of $G_u$. \\

Now assume $G_u[E_3]$ contains a copy of $K_3$, say $v_1v_2v_3$.  Let $x_i=|X(v_i)|$ for $i=1,2,3$. 
Again by (\ref{EQN: e1}) and Claim~\ref{CLM: 4.2}, we have 
$$3x_3\left(1-\frac{\sqrt{2}}{2}\right)n+3n^\frac{1}{2}\le x_1+x_2+x_3\le |I_0\cup I(v_1)\cup I(v_2)\cup I(v_3)|<|I_0|+3n^\frac 12.$$
Therefore, $|I_0|> 3(1-\frac{\sqrt{2}}{2})n$.
%{\color{red}By the claim above, their private vertex sets are disjoint. Now since $X_1\cup X_2\cup X_3\subset I_0\cup I(v_1) \cup I(v_2) \cup I(v_3)$, and $x_i\ge(1-\frac{\sqrt{2}}{2})n+n^\frac{1}{2}$ ($i=1,2,3$), we have $|I_0|\ge 3(1-\frac{\sqrt{2}}{2})n$,} 
Thus we have 
	$$d_1(u)=|E(G_u)|\le \binom{n-1-|I_0|}{2}<\binom{\frac{3\sqrt{2} -4}{2}n-1}{2}< \frac{3-2\sqrt{2}}{4}n^2+3n^\frac{3}{2}, $$
	a contradiction.
%    where the fomula $3n+n^\frac{3}{2}$ denote the number of edges from $E_1$, $E_2$. And this contradicts to the edge size of $G_u$; \\
%	\textbf{Subcase 2:} No such $v_3$ as described in \textbf{Subcase 1:} \\
%	By Turan's theorem \cite{ref6}, the number of edges in $G_u$ will less than $\frac{1}{4}(n-1-|I_0|)^2$, and similar to the discussion in \textbf{Subcase 1}, here $|I_0|\ge 2(1-\frac{\sqrt{2}}{2})n$, so we have 
%	$$|E(G_u)|\le \frac{1}{4}(n-1-|I_0|)^2+3n+n^\frac{3}{2} \le \frac{1}{4}\left( (\sqrt{2}-1)n-1\right) ^2+3n+n^\frac{3}{2}\le \frac{3-2\sqrt{2}}{4}n^2+3n^\frac{3}{2}, $$
%	and again, the fomula $3n+n^\frac{3}{2}$ here denote the number of edges from $E_1$, $E_2$. This also contradicts to the edge size of $G_u$. \\
	
This completes the proof of Theorem~\ref{THM: main2}. 
\end{proof}

%\section{Discussion and Remarks}

\end{document}